\documentclass{amsart}

\usepackage{amsmath}
\usepackage{amssymb}
\usepackage{amsthm}
\usepackage[pdftex]{graphicx}
\usepackage{tikz}

\newtheorem{thm}{Theorem}[section]
\newtheorem{prop}[thm]{Proposition}
\newtheorem{lem}[thm]{Lemma}

\theoremstyle{definition}
\newtheorem{dfn}[thm]{Definition}

\theoremstyle{remark}
\newtheorem{rem}[thm]{Remark}

\newcommand{\R}{\mathbb{R}}
\newcommand{\Z}{\mathbb{Z}}

\newcommand{\Mag}{\mathrm{Mag}}
\newcommand{\1}{\vec{1}}

\title[A proof for the positive definiteness of four point metric spaces]
{A direct proof for the positive definiteness of 
four point metric spaces}

\author{Kiyonori Gomi}

\address{
Department of Mathematics, 
Tokyo Institute of Technology, 
2-12-1 Ookayama, Meguro-ku, Tokyo, 152-8551, Japan.}

\email{kgomi@math.titech.ac.jp}



\begin{document}

\begin{abstract}
We provide a direct and elementary proof for the fact that every four point metric space is positive definite, which was first proved by Meckes based on some embedding theorems of metric spaces. As an outcome of the direct proof, we also provide a condition for the magnitude of a finite metric space to obey the inclusion-exclusion principle with respect to a specific choice of subspaces.
\end{abstract}

\maketitle

\tableofcontents


\section{Introduction}
\label{sec:introduction}

Let $X$ be a finite set, and $d$ a metric on $X$. The finite metric space $(X, d)$ is said to be \textit{positive definite} \cite{L1} if the \textit{similarity matrix} or the \textit{zeta matrix} of $(X, d)$
$$
\zeta_X = ( e^{-d(i, j)} )_{i, j \in X}
$$
is positive definite as a symmetric matrix. This property stems from the study of the \textit{magnitude} \cite{L1} of $(X, d)$. Intuitively, this is a numerical invariant which counts the number of points in $X$ taking the effect of the metric $d$. The magnitude of a positive definite metric space behaves nicely, and various conditions for the positive definiteness have been studied in \cite{L1,M}.

\medskip

It is clear that the $1$-point metric space is positive definite. In view of Sylvester's criterion, it is also clear that every $2$-point metric space is positive definite, because the determinant of the zeta matrix for $X_2 = \{ 1, 2 \}$ is
$$
\det \zeta_{X_2} = 
\begin{array}{|cc|}
1 & Z_{12} \\
Z_{21} & 1
\end{array}
=
1 - Z_{12}^2,
$$
where $Z_{ij} = e^{-d(i, j)}$ satisfies $1 - Z_{ij} > 0$ if $i \neq j$. We can see that $3$-point metric spaces $X_3 = \{ 1, 2, 3 \}$ are also positive definite in the same way using the following expression of the determinant \cite{L1} (Proposition 2.4.15)
\begin{align*}
\det \zeta_{X_3} 
&=
1 - Z_{12}^2 - Z_{13}^2 - Z_{23}^2
+ 2 Z_{12}Z_{13}Z_{23} \\
&=
(1 - Z_{12})(1 - Z_{13})(1 - Z_{23})
+ (1 - Z_{12})(Z_{12} - Z_{13}Z_{23}) \\
&\quad
+ (1 - Z_{13})(Z_{13} - Z_{12}Z_{23})
+ (1 - Z_{23})(Z_{23} - Z_{12}Z_{13}).
\end{align*}
Note that the triangle inequality is equivalent to $Z_{ij} - Z_{ik}Z_{kj} \ge 0$ for $i, j, k \in X$. There exists a $5$-point metric space which is not positive definite \cite{L1} (Example 2.2.7). Hence $n$-point metric spaces are generally not positive definite if $n \ge 5$. For $4$-point metric spaces, their positive definiteness is first established by Meckes \cite{M} (Theorem 3.6 (4)), where the method of the proof is to embed $4$-point metric spaces into a positive definite normed space. 

\medskip

It is plausible that one can show the positive definiteness of $4$-point metric spaces more directly without invoking an embedding theorem. However, such a proof seems to be not yet available in the literature. Then the purpose of this paper is to provide such a direct and elementary proof. 

The key to our proof is the formula for the determinant of the zeta matrix of the metric space $X_4 = \{ 1, 2, 3, 4 \}$ given by completing the square
$$
\det \zeta_{X_4} = - (1 - Z_{34}^2)(Z_{12} - b_0)^2 +
\frac{\Delta_{134}\Delta_{234}}{1 - Z_{34}^2},
$$
where $b_0$ is given by
$$
b_0 = 
\frac{Z_{13}Z_{23} + Z_{14}Z_{24} - Z_{14}Z_{23}Z_{34} - Z_{13}Z_{24}Z_{34}}
{1 - Z_{34}^2},
$$
and the determinants of the zeta matrices of the subspaces $\{ 1, 3, 4 \}$ and $\{ 2, 3, 4 \}$ are denoted by $\Delta_{134}$ and $\Delta_{234}$, respectively. Once the expression above is recognized, the elementary method can be applied to proving the positivity of $\det \zeta_{X_4}$, which leads to the positive definiteness of $X_4$.

\medskip

The key formula of $\det \zeta_{X_4}$ above seems to single out a particular metric on $X_4$, namely, one satisfying the equation $Z_{12} = b_0$. By a direct calculation, the magnitude $\Mag(X_4)$ of the metric space $X_4$ subject to $Z_{12} = b_0$ turns out to satisfy the so-called \textit{inclusion-exclusion principle} with respect to the subspaces $A = \{ 1, 3, 4 \}$ and $B = \{ 2, 3, 4 \}$: 
$$
\Mag(X_4) = \Mag(A) + \Mag(B) - \Mag(A \cap B).
$$
Furthermore, as will be established in Theorem \ref{thm:inclusion_exclusion}, this can be generalized to the $n$-point metric space $X_n = \{ 1, 2, \ldots, n \}$ with $n \ge 3$ and its subspaces $A = \{ 1, 3, 4, \ldots, n \}$ and $B = \{ 2, 3, 4, \ldots, n \}$. In addition, our condition $Z_{12} = b_0$ is generally not covered by a condition for the inclusion-exclusion principle which has been widely known \cite{L2,L1} (See Definition \ref{dfn:conditions}).

\medskip

It is natural to ask whether one can generalize the condition $Z_{12} = b_0$ for the inclusion-exclusion principle so as to be applicable to any choice of subspaces. We hope this problem to be solved in a future work.

\bigskip

The paper is organized as follows: In \S\ref{sec:proof} is presented our direct proof for the positive definiteness of four point metric spaces, after a few notations are introduced for the sake of clarity. \S\ref{sec:incl_excl} is devoted to the inclusion-exclusion principle under the condition $Z_{12} = b_0$. We start with a brief review of the magnitude of a finite metric space and the inclusion-exclusion principle widely known so far. Then the key formula is generalized, and our version of the inclusion-exclusion principle is established. The converse implication is also studied. Finally, the conditions for the inclusion-exclusion principle are compared.


\section{The proof}
\label{sec:proof}

\subsection{The positive definiteness}

The goal is to provide a direct proof of:

\begin{thm}[\cite{M}] \label{thm:meckes}
Every $4$-point metric space is positive definite.
\end{thm}

By Sylvester's criterion and the positive definiteness of $1$- $2$- and $3$-point metric spaces, Theorem \ref{thm:meckes} will follow  from:

\begin{thm} \label{thm:det}
For any $4$-point metric space, the determinant of its zeta matrix takes a positive value.
\end{thm}

We shall prove Theorem \ref{thm:det} in the remainder of this section.

\subsection{Preliminary}

We realize a $4$-point set as $X_4 = \{ 1, 2, 3, 4 \}$. Let $\mathcal{M}_4$ be the set of metrics $d$ on $X_4$. Through the map
$$
d \mapsto (e^{-d(1, 2)}, e^{-d(1, 3)}, e^{-d(1, 4)}
e^{-d(2, 3)}, e^{-d(2, 4)}, e^{-d(3, 4)}),
$$
this set $\mathcal{M}_4$ is bijective to the following subset in the open cube $(0, 1)^6 \subset \R^6$
$$
M_4 = 
\big\{ (Z_{ij})_{1 \le i < j \le 4} \in (0, 1)^6 \big|\ 
Z_{ij} - Z_{ik}Z_{kj} \ge 0 \ (\mbox{$i, j, k$ distinct})
\big\},
$$
where, for convenience, we put $Z_{ji} = Z_{ij}$ for $i < j$. It should be noted that we are considering genuine $4$-point metric spaces, so that the degeneracies, such as $d(1, 2) = 0$, are excluded. The determinant of the zeta matrix of $(X_4, d)$ is identified with the following polynomial function $\Delta$ in $Z = (Z_{ij})_{1 \le i < j \le 4} \in M_4$
\begin{align*}
\Delta 
&=
\begin{array}{|cccc|}
1 & Z_{12} & Z_{13} & Z_{14} \\
Z_{12} & 1 & Z_{23} & Z_{24} \\
Z_{13} & Z_{23} & 1 & Z_{34} \\
Z_{14} & Z_{24} & Z_{34} & 1
\end{array}
\\
&=
1 - Z_{12}^2 - Z_{13}^2 - Z_{14}^2 - Z_{23}^2 - Z_{24}^2 - Z_{34}^2 \\
& \quad
+ 2 Z_{12}Z_{13}Z_{23} + 2 Z_{12}Z_{14}Z_{24} 
+ 2 Z_{13}Z_{14}Z_{34} + 2 Z_{23}Z_{24}Z_{34} \\
& \quad
- 2 Z_{13}Z_{14}Z_{23}Z_{24} 
- 2 Z_{12}Z_{14}Z_{23}Z_{34} 
- 2 Z_{12}Z_{13}Z_{24}Z_{34} \\
&\quad
+ Z_{12}^2Z_{34}^2 + Z_{13}^2Z_{24}^2  + Z_{14}^2Z_{23}^2.
\end{align*}
Then Theorem \ref{thm:det} is equivalent to that $\Delta > 0$ on $M_4$.

\bigskip

We introduce some functions in $Z = (Z_{ij})_{1 \le i < j \le 4} \in M_4$. For any $i, j, k \in X_4$ such that $i < j < k$, we define a polynomial function $\Delta_{ijk}$ on $M_4$ by
$$
\Delta_{ijk}
= 1 - Z_{ij}^2 - Z_{ik}^2 - Z_{jk}^2
+ 2 Z_{ij}Z_{ik}Z_{jk}.
$$
This corresponds to the determinant of the zeta matrix of the $3$-point set $\{ i, j, k \} \subset X_4$ with the induced metric. Hence $\Delta_{ijk} > 0$ on $M_4$. For $Z \in M_4$, we define $b_\pm$ by
\begin{align*}
b_- &= \max \{ Z_{13}Z_{23}, Z_{14}Z_{24} \}, &
b_+ &= 
\min \bigg\{
\frac{Z_{23}}{Z_{13}},
\frac{Z_{13}}{Z_{23}},
\frac{Z_{24}}{Z_{14}},
\frac{Z_{14}}{Z_{24}}
\bigg\}.
\end{align*}
If $Z \in M_4$, then we have
$$
0 < b_- \le Z_{12} \le b_+ \le 1.
$$
For $Z \in M_4$, we also define $b_0$ by
$$
b_0 = 
\frac{Z_{13}Z_{23} + Z_{14}Z_{24} - Z_{14}Z_{23}Z_{34} - Z_{13}Z_{24}Z_{34}}
{1 - Z_{34}^2}.
$$

\begin{lem} \label{lem:b_0}
For $Z \in M_4$, we have $b_- \le b_0 \le b_+$.
\end{lem}

\begin{proof}
The inequality $b_- \le b_0$ follows from the following expressions of $b_0$
\begin{align*}
b_0 
&= 
Z_{13}Z_{23} + 
\frac{(Z_{14} - Z_{13}Z_{34})(Z_{24} - Z_{23}Z_{34})}{1 - Z_{34}^2} \\
&=
Z_{14}Z_{24} +
\frac{(Z_{13} - Z_{14}Z_{34})(Z_{23}- Z_{24}Z_{34})}{1 - Z_{34}^2}.
\end{align*}
To show $b_0 \le b_+$, we note that the following four cases can occur:
\begin{align*}
b_+ &= \frac{Z_{23}}{Z_{13}}, &
b_+ &= \frac{Z_{13}}{Z_{23}}, &
b_+ &= \frac{Z_{24}}{Z_{14}}, &
b_+ &= \frac{Z_{14}}{Z_{24}}.
\end{align*}
These cases are equivalent by permutations of points on $X_4$. For instance, the case $b_+ = Z_{23}/Z_{13}$ is transformed into the case $b_+ = Z_{13}/Z_{13}$ by the permutation $(12)$ exchanging the points $1$ and $2$. Similarly, the case $b_+ = Z_{23}/Z_{13}$ is transformed to the cases $b_+ = Z_{24}/Z_{14}$ and $b_+ = Z_{14}/Z_{24}$ by the permutations $(34)$ and $(12)(34)$, respectively. These permutations leave the function $b_0$ on $M_4$ invariant. Thus, it suffices to prove that $b_0 \le Z_{23}/Z_{13}$ for any $Z \in M_4$. 

Now, we make use of the following formula valid for all $Z \in M_4$
$$
\frac{Z_{23}}{Z_{13}} - b_0
=
\frac{
Z_{23}(1 - Z_{13}^2)(1 - Z_{34}^2)
- Z_{13}(Z_{14} - Z_{13}Z_{34})(Z_{24} - Z_{23}Z_{34})
}{Z_{13}(1 - Z_{34}^2)}.
$$
This is linear as a function in $Z_{24}$. If $Z \in M_4$, then $Z_{24}$ is subject to 
$$
Z_{24} \le \min\bigg\{ \frac{ Z_{23}}{Z_{34}}, \frac{Z_{34}}{Z_{23}} \bigg\}.
$$ 
We put $\tilde{b}_+ = \min\{ Z_{23}/Z_{34}, Z_{34}/Z_{23} \}$. With these preliminaries, we shall show that the function $f = Z_{23}/Z_{13} - b_0$ in $Z$ is non-negative on the following subset of the open cube $(0, 1)^6$
$$
\widetilde{M}_4 = \bigg\{ (Z_{ij})_{1 \le i < j \le 4} \in (0, 1)^6 \bigg|\
\begin{array}{l}
Z_{14} \ge Z_{13}Z_{34}, \ Z_{34} \ge Z_{13}Z_{14}, \\
Z_{23} \ge Z_{24}Z_{34}, \ Z_{34} \ge Z_{23}Z_{24}
\end{array}
\bigg\},
$$
which contains $M_4$ as a subset. If $Z \in \widetilde{M}_4$, then $Z_{24} \le \tilde{b}_+$. If $Z \in \widetilde{M}_4$ satisfies $Z_{24} = Z_{23}/Z_{34}$, then the value of $f$ at this $Z$ is
$$
f({ \scriptstyle Z_{24} = \frac{Z_{23}}{Z_{34}} })
=
\frac{Z_{23}(Z_{34} - Z_{13}Z_{14})}{Z_{13}Z_{34}} \ge 0.
$$
If $Z \in \widetilde{M}_4$ satisfies $Z_{24} = Z_{34}/Z_{23}$, then the value of $f$ at this $Z$ is
$$
f({\scriptstyle Z_{24} = \frac{Z_{34}}{Z_{23}} })
=
\frac{Z_{23}(Z_{34} - Z_{13}Z_{14})}{Z_{13}Z_{34}}
+
\frac{(Z_{23}^2 - Z_{34}^2)(Z_{14} - Z_{13}Z_{34})}
{Z_{23}Z_{34}(1 - Z_{34}^2)}.
$$
This is non-negative, because $1 > Z_{24} = Z_{34}/Z_{23}$ implies $Z_{23} > Z_{34}$. It follows that if $Z \in \widetilde{M}_4$ satisfies $Z_{24} = \tilde{b}_+$, then the value of $f$ at this $Z$ is non-negative. On $\widetilde{M}_4$, the function $f$ is decreasing in $Z_{24}$. Thus, for any $Z \in \widetilde{M}_4$, we have 
$$
Z' = (Z_{12}, Z_{13}, Z_{14}, Z_{23}, \tilde{b}_+, Z_{34})
\in \widetilde{M}_4,
$$
for which $f(Z) \ge f(Z') \ge 0$. Hence $f = Z_{23}/Z_{13} - b_0 \ge 0$ on $M_4 \subset \widetilde{M}_4$.
\end{proof}

\subsection{Proof}

We now prove that $\Delta > 0$ on $M_4$. The polynomial function $\Delta$ is quadratic in the variable $Z_{12}$. Completing the square, one has the expression
$$
\Delta = - (1 - Z_{34}^2)(Z_{12} - b_0)^2 +
\frac{\Delta_{134}\Delta_{234}}{1 - Z_{34}^2}.
$$
Our method of proof is a case-by-case estimate of $\Delta$ by using the expression above.  For $Z \in M_4$, the following two cases can occur:
\begin{itemize}
\item[($L$)]
$b_- \le Z_{12} \le b_0$,

\item[($R$)]
$b_0 \le Z_{12} \le b_+$.

\end{itemize}
The case ($L$) is further divided into two cases:
\begin{itemize}
\item[($L_1$)]
$b_- = Z_{13}Z_{23} \le Z_{12} \le b_0$,

\item[($L_2$)]
$b_- = Z_{14}Z_{24} \le Z_{12} \le b_0$,

\end{itemize}
and the case ($R$) is divided into four cases:
\begin{itemize}
\item[($R_1$)]
$b_0 \le Z_{12} \le b_+ = Z_{23}/Z_{13}$,

\item[($R_2$)]
$b_0 \le Z_{12} \le b_+ = Z_{13}/Z_{23}$,

\item[($R_3$)]
$b_0 \le Z_{12} \le b_+ = Z_{24}/Z_{14}$,

\item[($R_4$)]
$b_0 \le Z_{12} \le b_+ = Z_{14}/Z_{24}$.

\end{itemize}
The determinant of the zeta matrix is clearly invariant under the permutations of the points on $X_4$. Thus, for example, $\Delta > 0$ in the case ($L_1$) and that in ($L_2$) are equivalent, by the permutation $(34)$ exchanging the points $3$ and $4$ in $X_4$. As a result, it suffices to prove the claims $\Delta > 0$ in the cases ($L_1$)  and ($R_1$) only. These positivity claims are respectively shown in Proposition \ref{prop:L} and Proposition \ref{prop:R} below, by which the proof of Theorem \ref{thm:det} will be completed.

\bigskip

To show Proposition \ref{prop:L} and Proposition \ref{prop:R}, we prepare for lemmas.

\begin{lem} \label{lem:two_constraint_1}
If $Z = (Z_{ij}) \in M_4$ satisfies $Z_{14} = \frac{Z_{34}}{Z_{13}}$ and $Z_{23} = \frac{Z_{12}}{Z_{13}}$, then the value of $\Delta$ at this $Z$ is positive: 
$$
\Delta({\scriptstyle Z_{14} = \frac{Z_{34}}{Z_{13}}, Z_{23} = \frac{Z_{12}}{Z_{13}} })
> 0
$$
\end{lem}

\begin{proof}
Let $M'_4 \subset M_4$ be the subset
$$
M'_4 = \{ Z \in M_4 |\ Z_{13} Z_{14} = Z_{34}, Z_{13}Z_{23} = Z_{12} \}.
$$
The value of $\Delta$ at $Z \in M'_4$ can be expressed as
$$
\Delta({\scriptstyle Z_{14} = \frac{Z_{34}}{Z_{13}}, Z_{23} = \frac{Z_{12}}{Z_{13}}})
=
(1-Z_{13}^2) \bigg(
-\bigg( Z_{24} - \frac{Z_{12}Z_{34}}{Z_{13}} \bigg)^2
+\frac{(Z_{13}^2 - Z_{12}^2)(Z_{13}^2 - Z_{34}^2)}{Z_{13}^4}
\bigg).
$$
For $Z \in M'_4$, we let $b'_-$ and $b'_+$ be given by 
\begin{align*}
b'_- &= \max\{ Z_{12}Z_{14}, Z_{23}Z_{34} \}
= \frac{Z_{12}Z_{34}}{Z_{13}}, \\
b'_+ &=
\min\bigg\{ 
\frac{Z_{14}}{Z_{12}}, \frac{Z_{12}}{Z_{14}}, 
\frac{Z_{34}}{Z_{23}}, \frac{Z_{23}}{Z_{34}} \bigg\}
=
\min\bigg\{ 
\frac{Z_{34}}{Z_{12}Z_{13}}, \frac{Z_{12}Z_{13}}{Z_{34}},
\frac{Z_{13}Z_{34}}{Z_{12}}, \frac{Z_{12}}{Z_{13}Z_{34}} \bigg\} \\
&=
\min\bigg\{ 
\frac{Z_{12}Z_{13}}{Z_{34}},
\frac{Z_{13}Z_{34}}{Z_{12}} \bigg\}.
\end{align*}
If $Z \in M'_4$, then $b'_- \le Z_{24} \le b'_+$. Hence $\Delta({\scriptstyle Z_{14} = \frac{Z_{34}}{Z_{13}}, Z_{23} = \frac{Z_{12}}{Z_{13}}})$ is decreasing as a function in $Z_{24}$. As a result, for any $Z \in M'_4$, we have an element
$$
Z' = (Z_{12}, Z_{13}, Z_{14}, Z_{23}, b'_+, Z_{34}) \in M'_4,
$$
for which we have
$$
\Delta({\scriptstyle Z_{14} = \frac{Z_{34}}{Z_{13}}, Z_{23} = \frac{Z_{12}}{Z_{13}}})
= 
\Delta(Z)
\ge
\Delta(Z')
=
\Delta({\scriptstyle Z_{14} = \frac{Z_{34}}{Z_{13}}, Z_{23} = \frac{Z_{12}}{Z_{13}}, Z_{24} = b'_+}).
$$
Now, the proof will be completed by showing:
\begin{itemize}
\item[(a)]
In the case that $b'_+ = \frac{Z_{12}Z_{13}}{Z_{34}}$, we have $\Delta({\scriptstyle Z_{14} = \frac{Z_{34}}{Z_{13}}, Z_{23} = \frac{Z_{12}}{Z_{13}}, Z_{24} = b'_+}) > 0$.

\item[(b)]
In the case that $b'_+ = \frac{Z_{13}Z_{34}}{Z_{12}}$, we have $\Delta({\scriptstyle Z_{14} = \frac{Z_{34}}{Z_{13}}, Z_{23} = \frac{Z_{12}}{Z_{13}}, Z_{24} = b'_+}) > 0$.

\end{itemize}
These two cases turn out to be equivalent, because the permutation $(13)(24)$ on $X_4$ induces the following transformation on $M'_4 \subset M_4$
$$
(Z_{12}, Z_{13}, Z_{14}, Z_{23}, Z_{24}, Z_{34})
\leftrightarrow
(Z_{34}, Z_{13}, Z_{23}, Z_{14}, Z_{24}, Z_{12}).
$$
Hence it suffices to prove only (a): In the case that $b'_+ = Z_{12}Z_{13}/Z_{34}$, we have
\begin{multline*}
\Delta({\scriptstyle Z_{14} = \frac{Z_{34}}{Z_{13}}, Z_{23} = \frac{Z_{12}}{Z_{13}}, Z_{24} = \frac{Z_{12}Z_{13}}{Z_{34}} }) \\
=
\frac{(1 - Z_{13}^2)(Z_{13}^2 - Z_{34}^2)}{Z_{13}^4 Z_{34}^2}
(Z_{34}^2(Z_{13}^2 - Z_{12}^2) - Z_{12}^2 Z_{13}^2(Z_{13}^2 - Z_{34}^2)).
\end{multline*}
From $Z_{34}/Z_{13} = Z_{14} < 1$ and $Z_{12}/Z_{13} = Z_{23} < 1$, we respectively get 
\begin{align*}
Z_{34} &< Z_{13}, &
Z_{12} &< Z_{13}.
\end{align*}
From $Z_{12}Z_{13}/Z_{34} = b'_+ \le Z_{13}Z_{34}/Z_{12}$, we get $Z_{12}^2 \le Z_{34}^2$. Using the expression 
\begin{multline*}
Z_{34}^2(Z_{13}^2 - Z_{12}^2) - Z_{12}^2 Z_{13}^2(Z_{13}^2 - Z_{34}^2) \\
=
(Z_{34}^2 - Z_{12}^2)(Z_{13}^2 - Z_{12}^2 + Z_{12}^2Z_{13}^2)
+ Z_{12}^2 (1 - Z_{13}^2) (Z_{13}^2 - Z_{12}^2),
\end{multline*}
we find that $\Delta({\scriptstyle Z_{14} = \frac{Z_{34}}{Z_{13}}, Z_{23} = \frac{Z_{12}}{Z_{13}}, Z_{24} = \frac{Z_{12}Z_{13}}{Z_{34}} }) > 0$.
\end{proof}

\begin{lem} \label{lem:two_constraint_2}
If $Z = (Z_{ij}) \in M_4$ satisfies $Z_{13} = \frac{Z_{12}}{Z_{23}}$ and $Z_{14} = \frac{Z_{12}}{Z_{24}}$, then the value of $\Delta$ at this $Z$ is positive: 
$$
\Delta({\scriptstyle Z_{13} = \frac{Z_{12}}{Z_{23}}, Z_{14} = \frac{Z_{12}}{Z_{24}}}) > 0.
$$
\end{lem}

\begin{proof}
To begin with, we notice that the inequality
\begin{align*}
Z_{12} &< \min\{ Z_{13}, Z_{23}, Z_{14}, Z_{24} \}
\end{align*}
holds true if $Z \in M_4$ is subject to $Z_{13} = Z_{12}/Z_{23}$ and $Z_{14} = Z_{12}/Z_{24}$. We prove the present lemma by using the following expression valid for general $Z \in M_4$
$$
\Delta = - (1 - Z_{12}^2)(Z_{34} - c_0)^2 
+ \frac{\Delta_{123}\Delta_{124}}{1-Z_{12}^2},
$$
where $c_0$ is defined by
\begin{align*}
c_0 &= 
\frac{Z_{13}Z_{14} + Z_{23}Z_{24} - Z_{12}Z_{13}Z_{24} - Z_{12}Z_{14}Z_{23}}
{1 - Z_{12}^2}.
\end{align*}
We define $c_\pm$ by
\begin{align*}
c_- &= \max \{ Z_{13}Z_{14}, Z_{23}Z_{24} \}, &
c_+ &= \min \bigg\{
\frac{Z_{14}}{Z_{13}},
\frac{Z_{13}}{Z_{14}},
\frac{Z_{24}}{Z_{23}},
\frac{Z_{23}}{Z_{24}}
\bigg\}.
\end{align*}
If $Z \in M_4$, then $c_- \le Z_{34} \le c_+$. If $Z \in M_4$ satisfies $Z_{13} = Z_{12}/Z_{23}$ and $Z_{14} = Z_{12}/Z_{24}$, then we get
\begin{align*}
c_- &=  \max \bigg\{ \frac{Z_{12}^2}{Z_{23}Z_{24}}, Z_{23}Z_{24} \bigg\}, &
c_+ &= \min \bigg\{ \frac{Z_{24}}{Z_{23}}, \frac{Z_{23}}{Z_{24}} \bigg\}.
\end{align*}
Since $\Delta$ is quadratic in $Z_{34}$, if $Z \in M_4$ satisfies $c_0 \le Z_{34} \le c_+$, then we have
$$
Z' = (Z_{12}, Z_{13}, Z_{14}, Z_{23}, Z_{24}, c_+) \in M_4,
$$
for which $\Delta = \Delta(Z) \ge \Delta(Z') =  \Delta({\scriptstyle Z_{34} = c_+})$. Also, If $Z \in M_4$ satisfies $c_- \le Z_{34} \le c_0$, then $\Delta \ge \Delta({\scriptstyle Z_{34} = c_-})$. Imposing the constraints, we have:
\begin{itemize}
\item[(a)]
If $Z \in M_4$ satisfies $c_0 \le Z_{34} \le c_+$, $Z_{13} = Z_{12}/Z_{23}$ and $Z_{14} = Z_{12}/Z_{24}$, then
$$
\Delta({\scriptstyle Z_{13} = \frac{Z_{12}}{Z_{23}}, Z_{14} = \frac{Z_{12}}{Z_{24}}}) 
\ge
\Delta({\scriptstyle Z_{13} = \frac{Z_{12}}{Z_{23}}, Z_{14} = \frac{Z_{12}}{Z_{24}}, Z_{34} = c_+ }).
$$

\item[(b)]
If $Z \in M_4$ satisfies $c_- \le Z_{34} \le c_0$, $Z_{13} = Z_{12}/Z_{23}$ and $Z_{14} = Z_{12}/Z_{24}$, then
$$
\Delta({\scriptstyle Z_{13} = \frac{Z_{12}}{Z_{23}}, Z_{14} = \frac{Z_{12}}{Z_{24}}}) 
\ge
\Delta({\scriptstyle Z_{13} = \frac{Z_{12}}{Z_{23}}, Z_{14} = \frac{Z_{12}}{Z_{24}}, Z_{34} = c_- }).
$$

\end{itemize}
The proof will be completed by showing the positivity of the last terms.

\smallskip

(a) Suppose that $Z \in M_4$ satisfies  $c_0 \le Z_{34} \le c_+$, $Z_{13} = Z_{12}/Z_{23}$ and $Z_{14} = Z_{12}/Z_{24}$. Then there are two cases: $c_+ = Z_{24}/Z_{23}$ and $c_+ = Z_{23}/Z_{24}$. These two cases are equivalent by the permutation $(34)$ of the points $3$ and $4$ in $X_4$. Hence it suffices to study the case that $c_+ = Z_{24}/Z_{23}$. Then, we have
\begin{align*}
\Delta({\scriptstyle Z_{13} = \frac{Z_{12}}{Z_{23}}, Z_{14} = \frac{Z_{12}}{Z_{24}}, Z_{34} = c_+}) 
&=
\Delta({\scriptstyle Z_{13} = \frac{Z_{12}}{Z_{23}}, Z_{14} = \frac{Z_{12}}{Z_{24}}, Z_{34} = \frac{Z_{24}}{Z_{23}}}) \\
&=
\frac{(1 - Z_{23}^2)(Z_{24}^2 - Z_{12}^2)(Z_{23}^2 - Z_{24}^2)}
{Z_{23}^2Z_{24}^2} > 0,
\end{align*}
because $Z_{12}/Z_{24} = Z_{14} < 1$ and $Z_{24}/Z_{23} = Z_{34} < 1$.

(b) Suppose that $Z \in M_4$ satisfies  $c_- \le Z_{34} \le c_0$, $Z_{13} = Z_{12}/Z_{23}$ and $Z_{14} = Z_{12}/Z_{24}$. In the case that $c_- = Z_{23}Z_{24}$, we have
\begin{multline*}
\Delta({\scriptstyle Z_{13} = \frac{Z_{12}}{Z_{23}}, Z_{14} = \frac{Z_{12}}{Z_{24}}, Z_{34} = Z_{23}Z_{24}})
\\
=
\frac{(1 - Z_{23}^2)(1 - Z_{24}^2)}{Z_{23}^2Z_{24}^2}
\bigg(
Z_{23}^2 Z_{24}^2 - Z_{12}^2 Z_{23}^2 - Z_{12}^2 Z_{24}^2 
+ Z_{12}^2 Z_{23}^2 Z_{24}^2
\bigg).
\end{multline*}
From $c_- = Z_{23}Z_{24}$, one has $Z_{12} \le Z_{23}Z_{24}$. By the expression
\begin{multline*}
Z_{23}^2 Z_{24}^2 - Z_{12}^2 Z_{23}^3 - Z_{12}^2 Z_{24}^2 
+ Z_{12}^2 Z_{23}^2 Z_{24}^2 \\
=
(1 - (1- Z_{23}^2)(1- Z_{24}^2))(Z_{23}^2Z_{24}^2 - Z_{12}^2)
+ Z_{23}^2Z_{24}^2(1 - Z_{23}^2)(1 - Z_{24}^2),
\end{multline*}
we see $\Delta({\scriptstyle Z_{13} = \frac{Z_{12}}{Z_{23}}, Z_{14} = \frac{Z_{12}}{Z_{24}}, Z_{34} = Z_{23}Z_{24}}) > 0$. In the case that $c_- = Z_{12}^2/(Z_{23}Z_{24})$, we have $Z_{12} \ge Z_{23}Z_{24}$ from $Z_{12}^2/(Z_{23}Z_{24}) = c_- \ge Z_{23}Z_{24}$, and 
\begin{align*}
\Delta({\scriptstyle Z_{13} = \frac{Z_{12}}{Z_{23}}, Z_{14} = \frac{Z_{12}}{Z_{24}}, Z_{34} = \frac{Z_{12}^2}{Z_{23}Z_{24}} }) 
&=
\frac{
(Z_{23}^2 - Z_{12}^2)
(Z_{24}^2 - Z_{12}^2)
(1 - Z_{23}^2 - Z_{24}^2 + Z_{12}^2)
}
{Z_{23}^2Z_{24}^2} 
\\
&\ge
\frac{
(Z_{23}^2 - Z_{12}^2)
(Z_{24}^2 - Z_{12}^2)
(1 - Z_{23}^2 - Z_{24}^2 + Z_{23}^2Z_{24}^2)
}
{Z_{23}^2Z_{24}^2} 
\\
&=
\frac{
(Z_{23}^2 - Z_{12}^2)
(Z_{24}^2 - Z_{12}^2)
(1 - Z_{23}^2)(1 - Z_{24}^2)
}
{Z_{23}^2Z_{24}^2}, 
\end{align*}
which is also positive.
\end{proof}

\begin{lem} \label{lem:single_constraint}
If $Z = (Z_{ij}) \in M_4$ satisfies $Z_{12} = Z_{13}Z_{23}$, then the value of $\Delta$ at this $Z$ is positive: 
$$
\Delta({\scriptstyle Z_{12} = Z_{13}Z_{23}}) > 0.
$$
\end{lem}

\begin{proof}
We can express $\Delta({\scriptstyle Z_{12} = Z_{13}Z_{23}})$ as follows
$$
\Delta({\scriptstyle Z_{12} = Z_{13}Z_{23}})
= 
- (1 - Z_{23}^2)(Z_{14} - Z_{13}Z_{34})^2
+
(1 - Z_{13}^2)\Delta_{234},
$$
which can be thought of as a quadratic function in $Z_{14}$. If $Z \in M_4$, then
$$
\max \{ Z_{12}Z_{24}, Z_{13}Z_{34} \}
\le Z_{14}
\le
\min \bigg\{ 
\frac{Z_{24}}{Z_{12}}, 
\frac{Z_{12}}{Z_{24}},
\frac{Z_{34}}{Z_{13}}, 
\frac{Z_{13}}{Z_{34}} 
\bigg\}.
$$
Using $Z_{12} = Z_{13}Z_{23}$ and $Z_{23}Z_{34} \le Z_{24}$, one has
\begin{align*}
\min \bigg\{ 
\frac{Z_{24}}{Z_{12}}, 
\frac{Z_{12}}{Z_{24}},
\frac{Z_{34}}{Z_{13}}, 
\frac{Z_{13}}{Z_{34}} 
\bigg\}
&=
\min \bigg\{ 
\frac{Z_{24}}{Z_{13}Z_{23}}, 
\frac{Z_{13}Z_{23}}{Z_{24}},
\frac{Z_{34}}{Z_{13}}, 
\frac{Z_{13}}{Z_{34}}
\bigg\} \\
&=
\min \bigg\{ 
\frac{Z_{34}}{Z_{13}}, 
\frac{Z_{13}Z_{23}}{Z_{24}}
\bigg\}.
\end{align*}
Put $c_+ = \min\{ Z_{34}/Z_{13}, Z_{13}Z_{23}/Z_{24} \}$. Then the above expression of $\Delta({\scriptstyle Z_{12} = Z_{13}Z_{23}})$ leads to 
$$
\Delta({\scriptstyle Z_{12} = Z_{13}Z_{23}}) \ge 
\Delta({\scriptstyle Z_{12} = Z_{13}Z_{23}, Z_{14} = c_+ }).
$$
Now, the proof will be completed by showing the positivity of the last term. In the case that $c_+ = Z_{34}/Z_{13}$, we have
$$
\Delta({\scriptstyle Z_{12} = Z_{13}Z_{23}, Z_{14} = c_+ })
=
\Delta({\scriptstyle Z_{12} = Z_{13}Z_{23}, Z_{14} = \frac{Z_{34}}{Z_{13}}})
=
\Delta({\scriptstyle Z_{14} = \frac{Z_{34}}{Z_{13}}, Z_{23} = \frac{Z_{12}}{Z_{13}} }).
$$
This is positive by Lemma \ref{lem:two_constraint_1}. In the case that $c_+ = Z_{13}Z_{23}/Z_{24}$, we have
$$
\Delta({\scriptstyle Z_{12} = Z_{13}Z_{23}, Z_{14} = c_+})
=
\Delta({\scriptstyle Z_{12} = Z_{13}Z_{23}, Z_{14} = \frac{Z_{13}Z_{23}}{Z_{24}}})
=
\Delta({\scriptstyle Z_{13} = \frac{Z_{12}}{Z_{23}}, Z_{14} = \frac{Z_{12}}{Z_{24}}}).
$$
This is positive by Lemma \ref{lem:two_constraint_2}. 
\end{proof}

\begin{prop}[$L_1$] \label{prop:L}
We have $\Delta > 0$ if $Z = (Z_{ij}) \in M_4$ satisfies
$$
b_- = Z_{13}Z_{23} \le Z_{12} \le b_0.
$$
\end{prop}

\begin{proof}
Under the assumption of the proposition, $\Delta$ is increasing as a function in $Z_{12}$. Thus, if $Z \in M_4$, then we have
$$
Z' = (Z_{13}Z_{23}, Z_{13}, Z_{14}, Z_{23}, Z_{24}, Z_{34}) \in M_4,
$$
for which 
$$
\Delta = \Delta(Z) 
\ge \Delta(Z') = \Delta({\scriptstyle Z_{12} = Z_{13}Z_{23}}).
$$
The last term is positive by Lemma \ref{lem:single_constraint}.
\end{proof}

\begin{prop}[$R_1$] \label{prop:R}
We have $\Delta > 0$ if $Z = (Z_{ij}) \in M_4$ satisfies
$$
b_0 \le Z_{12} \le b_+ = \frac{Z_{23}}{Z_{13}}.
$$
\end{prop}

\begin{proof}
Under the present assumption, the function $\Delta$ is decreasing in $Z_{12}$. Thus, if $Z \in M_4$ satisfies $b_0 \le Z_{12} \le b_+ = Z_{23}/Z_{13}$, then we have
$$
Z' = (Z_{23}/Z_{13}, Z_{13}, Z_{14}, Z_{23}, Z_{24}, Z_{34})
\in M_4,
$$
for which we have
$$
\Delta = \Delta(Z) \ge 
\Delta(Z') = \Delta({\scriptstyle Z_{12} = \frac{Z_{23}}{Z_{13}}}).
$$
Note that the permutation $(13)$ exchanging the points $1$ and $3$ on $X_4$ transforms the constraint $Z_{23} = Z_{12}Z_{13}$ into $Z_{12} = Z_{13}Z_{23}$. Thus, by Lemma \ref{lem:single_constraint}, we see
$$
\Delta({\scriptstyle Z_{12} = \frac{Z_{23}}{Z_{13}}})
= 
\Delta({\scriptstyle Z_{23} = Z_{12}Z_{13} }) > 0,
$$
and the proof is completed.
\end{proof}


\section{The inclusion-exclusion principle}
\label{sec:incl_excl}

\subsection{Magnitude}

We here review the magnitude of a finite metric space, and its inclusion-exclusion principle \cite{L1}. 

\smallskip

Let $X = (X, d)$ be a finite metric space. As in \S\ref{sec:introduction}, the zeta matrix of $(X, d)$ is defined by 
$$
\zeta_X = ( e^{-d(i, j)})_{i, j \in X}
= (Z_{ij})_{i, j \in X},
$$
where $Z_{ij} = e^{-d(i, j)}$. We write $\Delta_X = \det \zeta_X$ for the determinant of the zeta matrix. A \textit{weighting} (or a \textit{weight} for short) on $X$ is a vector $\vec{w}_X \in \R^n$ such that $\zeta_X \vec{w}_X = \1$, where $n$ is the cardinality of $X$, and $\1 \in \R^n$ denotes the vector whose entries are $1$. In general, $X$ may not have a weight. If $\Delta_X \neq 0$, then $\vec{w}_X = \zeta_X^{-1}\1$ is the unique weight on $X$. When $X$ admits a weight $\vec{w}_X$, its magnitude $\Mag(X)$ is defined as the sum of the entries of $\vec{w}_X$. If $\Delta_X \neq 0$, then $\Mag(X)$ agrees with the sum of the entries of $\zeta_X^{-1}$. Using the standard inner product $\langle \ , \ \rangle$ on $\R^n$, we can express the magnitude as
$$
\Mag(X) = \langle \1, \vec{w}_X \rangle = \langle \1, \zeta_X^{-1} \1 \rangle.
$$
For example, we have
\begin{align*}
\Mag(\{ 1, 2 \}) &= \frac{2}{1 + Z_{12}}, &
\Mag(\{ 1, 2, 3 \}) &= 1 + \frac{2(1 - Z_{12})(1 - Z_{13})(1 - Z_{23})}{\Delta_{123}},
\end{align*}
where $\Delta_{123} = \Delta_{\{ 1, 2, 3 \}}$.

\bigskip

For the magnitude of $X$ to satisfy the inclusion-exclusion principle with respect to subspaces $A, B \subset X$, a condition has been known \cite{L1}. For clarity, we reformulate the condition in \cite{L1} as follows:

\begin{dfn} \label{dfn:conditions}
Let $A, B \subset X$ be subspaces. We define conditions (C1) and (C2) as follows:

\begin{enumerate}
\item[(C1)]
For any $a \in A$ and $b \in B$, there exists $c \in A \cap B$ such that
$$
d(a, b) = d(a, c) + d(c, b).
$$

\item[(C2)]
Either of the following holds true:
\begin{itemize}
\item[($\mathrm{C2}_a$)]
Any $a \in A$ admits $\pi(a) \in A \cap B$ such that: 
$$
d(a, c) = d(a, \pi(a)) + d(\pi(a), c)
$$ 
for all $c \in A \cap B$.

\item[($\mathrm{C2}_b$)]
Any $b \in A$ admits $\pi(b) \in A \cap B$ such that: 
$$
d(a, c) = d(b, \pi(b)) + d(\pi(b), c)
$$ 
for all $c \in A \cap B$.

\end{itemize}
\end{enumerate}
\end{dfn}

The condition ($\mathrm{C2}_a$) is termed ``$A$ projects to $A \cap B$'' in \cite{L2}, and ``$A \cap B$ is gated in $A$'' in \cite{B-K,D-S}. The condition ``$A$ projects to $B$'' in \cite{L1} is equivalent to (C1) and ($\mathrm{C2}_a$).

\medskip

Now, we state the inclusion-exclusion principle essentially due to Leinster:

\begin{thm}[\cite{L2,L1}] \label{thm:inclusion_exclusion_Leinster}
Let $A$ and $ B$ be subspaces of a finite metric space $X$ such that $A \cup B = X$ and $A \cap B \neq \emptyset$. Suppose that:
\begin{itemize}
\item
$\Mag(A)$ and $\Mag(B)$ are defined; and 

\item
(C1) and (C2) are satisfied.

\end{itemize}
Then $\Mag(A \cap B)$ and $\Mag(X_n)$ are defined, and
$$
\Mag(X_n) = \Mag(A) + \Mag(B) - \Mag(A \cap B).
$$
\end{thm}

The statement above is not exactly the same as that in \cite{L1} (Proposition 2.3.2), but can be shown by applying arguments in \cite{L2,L1}.

\subsection{A generalization of the key formula}

Let us realize an $n$-point set as $X_n = \{ 1, 2, \ldots, n \}$, and consider a metric $d$ on $X_n$. The key formula for the direct proof in \S\ref{sec:proof} was the expression of $\Delta_{X_4}$ given by completing the square. We here provide its generalization to $\Delta_{X_n}$. For its description, we start with a lemma to be used repeatedly:

\begin{lem} \label{lem:formula_for_determinant}
Suppose $n \ge 2$. For a real number $x \in \R$, vectors $\vec{a}, \vec{b} \in \R^{n-1}$, and a symmetric matrix $M$ of size $n-1$, we consider a square matrix of size $n$
$$
\left(
\begin{array}{cc}
x & {}^t\vec{a} \\
\vec{b} & M
\end{array}
\right),
$$
where ${}^t( \ )$ stands for the transpose. Then the determinant of the above matrix is
$$
\begin{array}{|cc|}
x & {}^t\vec{a} \\
\vec{b} & M
\end{array}
=
x \lvert M \rvert - \langle \vec{a}, \widetilde{M}\vec{b} \rangle,
$$
where $\widetilde{M}$ is the adjugate matrix (the transpose of the cofactor) of $M$. Thus, in particular, if $\lvert M \rvert \neq 0$, then $M^{-1} = \widetilde{M}/\lvert M \rvert$ and 
$$
\begin{array}{|cc|}
x & {}^t\vec{a} \\
\vec{b} & M
\end{array}
= \lvert M \rvert (x - \langle \vec{a}, M^{-1}\vec{b} \rangle).
$$
\end{lem}

\begin{proof}
By the linearity of the determinant in the first column, we get
$$
\begin{array}{|cc|}
x & {}^t\vec{a} \\
\vec{b} & M
\end{array}
=
\begin{array}{|cc|}
x & {}^t\vec{a} \\
0 & M
\end{array}
+
\begin{array}{|cc|}
0 & {}^t\vec{a} \\
\vec{b} & M
\end{array}
=
x \lvert M \rvert + 
\begin{array}{|cc|}
0 & {}^t\vec{a} \\
\vec{b} & M
\end{array}.
$$
In view of Cramer's formula, the last term is identified with $- \langle \vec{a}, \widetilde{M}\vec{b} \rangle$.
\end{proof}

\begin{dfn} \label{dfn:b_0}
For $n \ge 3$, let $A, B \subset X_n = \{ 1, 2, \ldots, n \}$ be the subspaces
\begin{align*}
A &= \{ 1, 3, 4, \ldots, n \}, &
B &= \{ 2, 3, 4, \ldots, n \}.
\end{align*}
Assuming $\Delta_{A \cap B} \neq 0$, we define $b_-$ and $b_0$ by
\begin{align*}
b_- &= \max\{ Z_{1j}Z_{2j} |\ j = 3, 4, \ldots, n \}, \\
b_0 &= 
- \frac{1}{\Delta_{A \cap B}}
\begin{array}{|cc|}
0 & {}^t\vec{a} \\
\vec{b} & \zeta_{A \cap B}
\end{array}
=
- \frac{1}{\Delta_{A \cap B}}
\begin{array}{|cccccc|}
0 & Z_{13} & Z_{14} & Z_{15} & \cdots & Z_{1n} \\
Z_{23} & 1 & Z_{34} & Z_{35} & \cdots & Z_{3n} \\
Z_{24} & Z_{34} & 1 & Z_{45} & \cdots & Z_{4n} \\
Z_{25} & Z_{35} & Z_{45} & 1 & \cdots & Z_{5n} \\
\vdots & \vdots & \vdots & \vdots & \ddots & \vdots \\
Z_{2n} & Z_{3n} & Z_{4n} & Z_{5n} & \cdots & 1
\end{array},
\end{align*}
where vectors $\vec{a}, \vec{b} \in \R^{n-2}$ are given by
\begin{align*}
{}^t\vec{a} &= (Z_{13}, Z_{14}, \ldots, Z_{1n}), &
{}^t\vec{b} &= (Z_{23}, Z_{24}, \ldots, Z_{2n}).
\end{align*}
\end{dfn}

By Lemma \ref{lem:formula_for_determinant}, we have
\begin{align*}
\Delta_A &= 
\begin{array}{|cc|}
1 & {}^t\vec{a} \\
\vec{a} & \zeta_{A \cap B}
\end{array}
= 
\Delta_{A \cap B} 
- \langle \vec{a}, \widetilde{\zeta}_{A \cap B} \vec{a} \rangle
=
\Delta_{A \cap B} 
(1 - \langle \vec{a}, \zeta_{A \cap B}^{-1}\vec{a} \rangle),
\end{align*}
where $\Delta_{A \cap B} \neq 0$ is assumed for the last equality. Replacing $\vec{a}$ in the above with $\vec{b}$, we get the corresponding formula for $\Delta_B$. We also have
$$
b_0 = \langle \vec{a}, \zeta_{A \cap B}^{-1}\vec{b} \rangle.
$$

\smallskip

\begin{lem} \label{lem:adjugate}
In the setup of Definition \ref{dfn:b_0}, we have
$$
\widetilde{\zeta}_B
=
\Delta_{A \cap B}
\left(
\begin{array}{cc}
1 & - {}^t(\zeta_{A \cap B}^{-1}\vec{b}) \\
- \zeta_{A \cap B}^{-1} \vec{b} & 
\frac{\Delta_B}{\Delta_{A \cap B}} \zeta_{A \cap B}^{-1} 
+ (\zeta_{A \cap B}^{-1}\vec{b}) {}^t(\zeta_{A \cap B}^{-1}\vec{b}) 
\end{array}
\right).
$$
\end{lem}

\begin{proof}
Using $x \in \R$, $\vec{y} \in \R^{n-2}$ and a symmetric matrix $N$ of size $n-2$, we can express $\widetilde{\zeta}_B$ as
$$
\widetilde{\zeta}_B
=
\left(
\begin{array}{cc}
x & {}^t\vec{y} \\
\vec{y} & N
\end{array}
\right).
$$
The adjugate of $\zeta_B$ is subject to $\zeta_B \widetilde{\zeta}_B = \Delta_B E$, where $E$ is the identity matrix. Under the assumption $\Delta_{A \cap B} \neq 0$, this equation in $x, \vec{y}, N$ is uniquely solved.
\end{proof}

\begin{prop} \label{prop:determinant_general}
In the setup of Definition \ref{dfn:b_0}, we have
$$
\Delta_{X_n} = - \Delta_{A \cap B}(Z_{12} - b_0)^2 
+ \frac{\Delta_A \Delta_B}{\Delta_{A \cap B}}.
$$
\end{prop}

\begin{proof}
By definition, we have
$$
\Delta_{X_n} =
\begin{array}{|ccc|}
1 & Z_{12} & {}^t\vec{a} \\
Z_{12} & 1 & {}^t\vec{b} \\
\vec{a} & \vec{b} & \zeta_{A \cap B}
\end{array}.
$$
Taking the derivative with respect to $Z_{12}$, we readily see
$$
\Delta_{X_n} = - \Delta_{A \cap B}(Z_{12} - b_0)^2
+ \Delta_{A \cap B} b_0^2 + \Delta({\scriptstyle Z_{12} = 0 }).
$$
Therefore the proof will be completed by showing that $\Delta_{A \cap B} b_0^2 + \Delta({\scriptstyle Z_{12} = 0 })$ agrees with $\Delta_A \Delta_B/\Delta_{A \cap B}$. Lemma \ref{lem:formula_for_determinant} leads to
$$
\Delta({\scriptstyle Z_{12} = 0})
=
\begin{array}{|ccc|}
1 & 0 & {}^t\vec{a} \\
0 & 1 & {}^t\vec{b} \\
\vec{a} & \vec{b} & \zeta_{A \cap B}
\end{array}
=
\Delta_B - 
\big\langle 
\left(
\begin{array}{c}
0 \\ \vec{a}
\end{array}
\right),
\widetilde{\zeta}_B
\left(
\begin{array}{c}
0 \\ \vec{a}
\end{array}
\right)
\big\rangle.
$$
Lemma \ref{lem:adjugate} then allows us to have
\begin{align*}
\big\langle 
\left(
\begin{array}{c}
0 \\ \vec{a}
\end{array}
\right),
\widetilde{\zeta}_B
\left(
\begin{array}{c}
0 \\ \vec{a}
\end{array}
\right)
\big\rangle
&=
\big\langle \vec{a},
\Delta_B \zeta_{A \cap B}^{-1} \vec{a}
+ 
\Delta_{A \cap B}
(\zeta_{A \cap B}^{-1}\vec{b}) {}^t(\zeta_{A \cap B}^{-1}\vec{b})\vec{a}
\big\rangle \\
&=
\Delta_B
\langle \vec{a}, \zeta_{A \cap B}^{-1}\vec{a} \rangle
+
\Delta_{A \cap B}
\big\langle 
\vec{a}, 
\langle \zeta_{A \cap B}^{-1}\vec{b}, 
\vec{a} \rangle \zeta^{-1}_{A \cap B} \vec{b}
\big\rangle \\
&=
\Delta_B
\bigg(
1 - \frac{\Delta_A}{\Delta_{A \cap B}}
\bigg)
+
\Delta_{A \cap B}
\langle \zeta_{A \cap B}^{-1} \vec{b}, \vec{a} \rangle 
\langle \vec{a}, \zeta_{A \cap B}^{-1} \vec{b} \rangle \\
&=
\Delta_B - \frac{\Delta_{A}\Delta_B}{\Delta_{A \cap B}}
+ \Delta_{A \cap B} b_0^2.
\end{align*}
Hence $\Delta({\scriptstyle Z_{12} = 0 }) = \Delta_A\Delta_B/\Delta_{A \cap B} - \Delta_{A \cap B} b_0^2$, and the proof is completed.
\end{proof}

\subsection{The inclusion-exclusion principle}

We now show the inclusion-exclusion principle under the condition $Z_{12} = b_0$.

\begin{lem} \label{lem:magnitude_B}
For $n \ge 3$, let $A, B \subset X_n = \{ 1, 2, \ldots, n \}$ be the subspaces
\begin{align*}
A &= \{ 1, 3, 4, \ldots, n \}, &
B &= \{ 2, 3, 4, \ldots, n \}.
\end{align*}
Suppose that $\Delta_B \Delta_{A \cap B} \neq 0$. Then the magnitude of $B$ is expressed as
$$
\Mag(B) = 
\Mag(A \cap B) 
+ \frac{\Delta_{A \cap B}}{\Delta_B}
(1 - \langle \vec{w}_{A \cap B}, \vec{b} \rangle)^2,
$$
where $\vec{w}_{A \cap B}$ is the weight on $A \cap B$. Also the weight $\vec{w}_B$ on $B$ is expressed as
\begin{align*}
\vec{w}_B &= 
\left(\begin{array}{c} 
\beta \\ 
\vec{w}'_B 
\end{array}
\right)
,
&
&
\left\{
\begin{array}{l}
\beta 
= \frac{\Delta_{A \cap B}}{\Delta_B}
(1 - \langle \vec{b}, \vec{w}_{A \cap B} \rangle), \\
\vec{w}'_B 
= \vec{w}_{A \cap B} - \beta \zeta_{A \cap B}^{-1}\vec{b}.
\end{array}
\right.
\end{align*}
\end{lem}

\begin{proof}
The expressions directly follow from Lemma \ref{lem:adjugate}.
\end{proof}

\begin{lem} \label{lem:defect}
For $n \ge 3$, let $A, B \subset X_n = \{ 1, 2, \ldots, n \}$ be the subspaces
\begin{align*}
A &= \{ 1, 3, 4, \ldots, n \}, &
B &= \{ 2, 3, 4, \ldots, n \}.
\end{align*}
Suppose $\Delta_{X_n} \Delta_A \Delta_B \Delta_{A \cap B} \neq 0$, and define $\alpha, \beta \in \R$ by
\begin{align*}
\alpha 
&= \frac{\Delta_{A \cap B}}{\Delta_A}
(1 - \langle \vec{a}, \vec{w}_{A \cap B} \rangle), &
\beta
&= \frac{\Delta_{A \cap B}}{\Delta_B}
(1 - \langle \vec{b}, \vec{w}_{A \cap B} \rangle), &
\end{align*}
Then we have
\begin{align*}
&\Mag(X_n) - \Mag(A) - \Mag(B) + \Mag(A \cap B) \\
&\quad
=
\frac{b_0 - Z_{12}}{\Delta_{X_n}}
\bigg(
(b_0 - Z_{12})\Delta_A \alpha^2
+
(b_0 - Z_{12})\Delta_B \beta^2 
+
2 \frac{\Delta_A\Delta_B}{\Delta_{A \cap B}}\alpha \beta
\bigg).
\end{align*}
\end{lem}

\begin{proof}
By Lemma \ref{lem:adjugate}, we can express the inverse of $\zeta_{X_n}$ as
$$
\zeta_{X_n}^{-1}
=
\left(
\begin{array}{ccc}
x & s & {}^t\vec{p} \\
s & y & {}^t\vec{q} \\
\vec{p} & \vec{q} & N
\end{array}
\right),
$$
where $x, y, s \in \R$ are 
\begin{align*}
x &= \frac{\Delta_B}{\Delta_{X_n}}, &
y &= \frac{\Delta_A}{\Delta_{X_n}}, &
s &= -(Z_{12} - b_0)\frac{\Delta_{A \cap B}}{\Delta_{X_n}},
\end{align*}
the vectors $\vec{p}, \vec{q} \in \R^{n-2}$ and the symmetric matrix $N$ of size $n - 2$ are 
\begin{align*}
\vec{p} &= 
- \frac{\Delta_B}{\Delta_{X_n}} \zeta_{A \cap B}^{-1}\vec{a}
+ (Z_{12} - b_0)
\frac{\Delta_{A \cap B}}{\Delta_{X_n}}\zeta_{A \cap B}^{-1}\vec{b}, \\
\vec{q} &=
(Z_{12} - b_0)
\frac{\Delta_{A \cap B}}{\Delta_{X_n}}\zeta_{A \cap B}^{-1}\vec{a}
- \frac{\Delta_A}{\Delta_{X_n}} \zeta_{A \cap B}^{-1}\vec{b}, \\
N &= \zeta_{A \cap B}^{-1} 
+ \frac{\Delta_B}{\Delta_{X_n}}
(\zeta_{A \cap B}^{-1}\vec{a}) {}^t(\zeta_{A \cap B}^{-1}\vec{a})
+ \frac{\Delta_A}{\Delta_{X_n}}
(\zeta_{A \cap B}^{-1}\vec{b}) {}^t(\zeta_{A \cap B}^{-1}\vec{b}) \\
&\quad
- (Z_{12} - b_ 0) \frac{\Delta_{A \cap B}}{\Delta_{X_n}}
\big(
(\zeta_{A \cap B}^{-1}\vec{a}) {}^t(\zeta_{A \cap B}^{-1}\vec{b}) +
(\zeta_{A \cap B}^{-1}\vec{b}) {}^t(\zeta_{A \cap B}^{-1}\vec{a})
\big).
\end{align*}
Using the above expression and Lemma \ref{lem:magnitude_B}, we have
\begin{align*}
\Mag(X_n)
&= \Mag(A \cap B) 
+ \frac{\Delta_A\Delta_B}{\Delta_{X_n}\Delta_{A \cap B}}
(\Mag(A) + \Mag(B) - 2 \Mag(A \cap B)) \\
&\quad
-2(Z_{12}- b_0) \frac{\Delta_{A \cap B}}{\Delta_{X_n}}
(1 - \langle \vec{w}_{A \cap B}, \vec{a} \rangle)
(1 - \langle \vec{w}_{A \cap B}, \vec{b} \rangle).
\end{align*}
This formula and Proposition \ref{prop:determinant_general} lead to
\begin{align*}
\Mag(X_n)
&= 
\Mag(A) + \Mag(B) - \Mag(A \cap B) \\
&\quad
+
(Z_{12} - b_0)^2\frac{\Delta_{A \cap B}}{\Delta_{X_n}}
(\Mag(A) + \Mag(B) - 2 \Mag(A \cap B)) \\
&\quad
-2(Z_{12} - b_0) \frac{\Delta_{A \cap B}}{\Delta_{X_n}}
(1 - \langle \vec{w}_{A \cap B}, \vec{a} \rangle)
(1 - \langle \vec{w}_{A \cap B}, \vec{b} \rangle).
\end{align*}
Using Lemma \ref{lem:magnitude_B} again, we complete the proof. 
\end{proof}

\begin{thm} \label{thm:inclusion_exclusion}
For $n \ge 3$, let $A, B \subset X_n = \{ 1, 2, \ldots, n \}$ be the subspaces
\begin{align*}
A &= \{ 1, 3, 4, \ldots, n \}, &
B &= \{ 2, 3, 4, \ldots, n \}.
\end{align*}
Suppose that: 
\begin{itemize}
\item
$\Delta_{A} \Delta_{B} \Delta_{A \cap B} \neq 0$, and

\item
$Z_{12} = b_0$.

\end{itemize}
Then $\Mag(X_n)$ is defined, and
$$
\Mag(X_n) = \Mag(A) + \Mag(B) - \Mag(A \cap B).
$$
\end{thm}

\begin{proof}
By Proposition \ref{prop:determinant_general}, we have $\Delta_{X_n} \neq 0$, and the magnitude of $X_n$ is defined. Then Lemma \ref{lem:defect} completes the proof.
\end{proof}

Under some assumptions, we can also show the converse of Theorem \ref{thm:inclusion_exclusion}:

\begin{thm} \label{thm:converse}
For $n \ge 3$, let $A, B \subset X_n = \{ 1, 2, \ldots, n \}$ be the subspaces
\begin{align*}
A &= \{ 1, 3, 4, \ldots, n \}, &
B &= \{ 2, 3, 4, \ldots, n \}.
\end{align*}
Suppose $\Delta_{X_n}\Delta_A\Delta_B\Delta_{A \cap B} \neq 0$, and define $\delta \in \R$ by
$$
\delta = \Mag(X_n) - \Mag(A) - \Mag(B) + \Mag(A \cap B).
$$
\begin{enumerate}
\item[(a)]
In the case that $n = 3$, $\delta = 0$ implies $Z_{12} = b_0$.

\item[(b)]
In the case that $n \ge 4$, we assume
\begin{itemize}
\item[(i)]
$X_n$ is positive definite;

\item[(ii)]
$A$ and $B$ are positive weighting; and

\item[(iii)]
$Z_{12} \le b_0$.

\end{itemize}
Then $\delta = 0$ implies $Z_{12} = b_0$.

\end{enumerate}
\end{thm}

Note that a metric space is said to be \textit{positive weighting} if it admits a weight whose components are positive \cite{L1}. It is known that metric spaces consisting of less than or equal to three points are positive weighting \cite{L1} (Proposition 2.4.15). As reviewed in \S\ref{sec:introduction}, the metric space $X_n$ is positive definite if $n \le 4$. Thus, in the case of $n = 4$, the assumptions (i) and (ii) are redundant.

\begin{proof}[Proof of Theorem \ref{thm:converse}]
(a) 
In the case that $n = 3$, we have an expression
$$
\delta
=
-\frac{2(Z_{12} - b_0)(\Delta_{X_3} - (1-Z_{12})(Z_{12} - Z_{13}Z_{23}))}
{(1 + Z_{12})(1 + Z_{13})\Delta_{X_3}}.
$$
Then $\delta = 0$ implies $Z_{12} = b_0$, since the following factor is positive:
\begin{align*}
&\Delta_{X_3}  - (1-Z_{12})(Z_{12} - Z_{13}Z_{23}) \\
&=
(1 - Z_{12})(1 - Z_{13})(1 - Z_{23}) \\
&\quad
+ (1 - Z_{13})(Z_{13} - Z_{12}Z_{23})
+ (1 - Z_{23})(Z_{23} - Z_{12}Z_{13}).
\end{align*}

(b) 
In general, if a finite metric space is positive definite, then so is its subspace \cite{L1} (Lemma 2.4.2). Hence $\Delta_A$, $\Delta_B$ and $\Delta_{A \cap B}$ are positive by (i). By Lemma \ref{lem:magnitude_B} and (ii), we also see that $\alpha$ and $\beta$ are positive. Thus, further assuming (iii), we get the positivity of the factor
$$
(b_0 - Z_{12})\Delta_A \alpha^2
+
(b_0 - Z_{12})\Delta_B \beta^2 
+
2 \frac{\Delta_A\Delta_B}{\Delta_{A \cap B}}\alpha \beta
$$
in the formula of $\delta$ in Lemma \ref{lem:defect}. Therefore $\delta = 0$ implies $Z_{12} = b_0$. 
\end{proof}

\medskip

\begin{rem}
Suppose $n \ge 2$. For $i = 0, \ldots, n-1$, we define a subspace $X^{(i)}$ of the $n$-point metric space $X_n = \{ 1, 2, \ldots, n \}$ by $X^{(i)} = \{ i+1, i + 2, \ldots, n \}$. For $i = 1, \ldots, n-1$, we also define a vector $\vec{x}_i \in \R^{n - i}$ by ${}^t\vec{x}_i = (Z_{ij})_{j = i+1}^n$. If $\Delta_{X^{(i)}} \neq 0$ for $i = 0, 1, \cdots$, then Lemma \ref{lem:magnitude_B} leads to the following formula for the magnitude of $X_n = X^{(0)}$
\begin{align*}
\Mag(X^{(0)}) 
&= \frac{\Delta_{X^{(1)}}}{\Delta_{X^{(0)}}}
(1 - \langle \vec{w}_{X^{(1)}}, \vec{x}_1 \rangle)^2 + \Mag(X^{(1)}) \\
&=
\sum_{i = 1}^{n-1}
 \frac{\Delta_{X^{(i)}}}{\Delta_{X^{(i-1)}}}
(1 - \langle \vec{w}_{X^{(i)}}, \vec{x}_i \rangle)^2 + 1,
\end{align*}
where $\vec{w}_{X^{(i)}} = \zeta_{X^{(i)}}^{-1}\1 \in \R^{n-i}$ is the weight on $X^{(i)}$. A consequence of the formula above is another proof for the inequalities $\Mag(X) \ge \Mag(Y) \ge 1$ valid for any non-empty subspace $Y$ in a positive definite finite metric space $X$, which have been shown in \cite{L1} (Corollary 2.4.4, Corollary 2.4.5).
\end{rem}

\subsection{Comparison of conditions}

We compare the condition for the inclusion-exclusion principle in Theorem \ref{thm:inclusion_exclusion} with that in Theorem \ref{thm:inclusion_exclusion_Leinster} applied to our choice of the subspaces.

\begin{lem} \label{lem:comparison}
For $n \ge 3$, let $A$ and $B$ be the following subspaces of $(X_n, d)$
\begin{align*}
A &= \{ 1, 3, 4, \ldots, n \}, &
B &= \{ 2, 3, 4, \ldots, n \}.
\end{align*}

\begin{itemize}
\item[(a)]
(C1) is equivalent to $b_- = Z_{12}$.

\item[(b)]
Suppose $\Delta_{A \cap B} \neq 0$. Then (C2) implies $b_- = b_0$.

\end{itemize}
\end{lem}

\begin{proof}
The equivalence in (a) is clear, so we prove (b). Note that $b_-$ and $b_0$ are invariant under the permutation $(12)$ exchanging $1, 2 \in X_n$, while this permutation exchanges the conditions ($\mathrm{C1}_a$) and ($\mathrm{C1}_b$). Therefore it suffices to show that ($\mathrm{C1}_a$) implies $b_- = b_0$. Note also that $b_-$ and $b_0$ are also invariant under any permutations of $A \cap B$. Thus, in ($\mathrm{C1}_a$), we can assume that $\pi(1) = 3$. In this case, ($\mathrm{C1}_a$) is equivalent to the triangle equalities $Z_{1j} = Z_{13}Z_{3j}$ for $4 \le j \le n$. Then the triangle inequality $Z_{23} \ge Z_{2j}Z_{3j}$ leads to 
$$
Z_{1j}Z_{2j} = Z_{13}Z_{3j}Z_{2j} \le Z_{13}Z_{23}.
$$
Therefore $b_- = Z_{13}Z_{23}$. Now, noting the relation of vectors
$$
(0, Z_{13}, Z_{14}, Z_{15}, \ldots, Z_{1n})
= Z_{13}(0, 1, Z_{34}, Z_{35}, \ldots, Z_{3n}),
$$
we can directly verify $b_0 = Z_{13}Z_{23}$.
\end{proof}

In the case that $n = 3$, we always have $b_- = b_0 = Z_{13}Z_{23}$. Hence no difference arises. In the case that $n = 4$, we have $b_- \le b_0$ by Lemma \ref{lem:b_0}. Therefore there exists a metric on $X_4$ such that $b_- < b_0 = Z_{12}$. (An example is provided in Remark \ref{rem:Mayer_Vietoris} below.) In this case, our inclusion-exclusion principle in Theorem \ref{thm:inclusion_exclusion} is not covered by that in Theorem \ref{thm:inclusion_exclusion_Leinster}.

\begin{rem} \label{rem:converse_to_C2_implication}
The converse of Lemma \ref{lem:comparison} (b) holds true if $n = 3, 4$: This is clear in the case of $n = 3$. In the case of $n = 4$, we can see that $b_- = b_0$ implies (C2) by using the expressions of $b_0$ proving $b_- \le b_0$ in the proof of Lemma \ref{lem:b_0}. In the case of $n = 5$, there exists an example such that the converse does not hold. 

To describe the example, let $S^1 \subset \R^2$ be the unit circle centred at the origin. By the geodesic distance, $S^1$ gives rise to a metric space. We write $p(\theta) = (\cos \theta, \sin \theta) \in S^1$ for the point specified by $\theta \in \R/2\pi \Z$. As in Figure \ref{fig:1} (left), let  $X_5 = \{ 1, 2, 3, 4, 5 \} \subset S^1$ be the subspace consisting of 
\begin{align*}
1 &= p(0), &
2 &= p(3\pi/4), &
3 &= p(\pi/2), &
4 &= p(-\pi/2), &
5 &= p(\pi).
\end{align*}
In the case of $n = 5$, we can generally express $b_0$ as
$$
b_0 = Z_{13}Z_{23} + \frac{
(Z_{14} - Z_{13}Z_{34})P + (Z_{15}-Z_{13}Z_{35})Q
}{\Delta_{345}},
$$
where $P$ and $Q$ are 
\begin{align*}
P
&=
(1 - Z_{35}^2)(Z_{24} - Z_{23}Z_{34})
- (Z_{25} - Z_{23}Z_{35})(Z_{45} - Z_{34}Z_{35}) \\
&=
(1 - Z_{35}^2)(Z_{24} - Z_{25}Z_{45})
- (Z_{23} - Z_{25}Z_{35})(Z_{34} - Z_{35}Z_{45}), \\
Q
&=
(1 - Z_{34}^2)(Z_{25} - Z_{23}Z_{35})
- (Z_{24} - Z_{23}Z_{34})(Z_{45} - Z_{34}Z_{35}) \\
&=
(1 - Z_{34}^2)(Z_{25} - Z_{24}Z_{45})
- (Z_{23} - Z_{24}Z_{34})(Z_{35} - Z_{34}Z_{45}).
\end{align*}
For the present example, we have $Z_{12} = Z_{13}Z_{23}$, so that $b_- = Z_{13}Z_{23}$. From $Z_{24} = Z_{25}Z_{45}$ and $Z_{34} = Z_{35}Z_{45}$, it follows that $P = 0$. Since $Z_{15} = Z_{13}Z_{35}$, we find $b_0 = Z_{13}Z_{23} = b_-$. However, (C2) is not satisfied.
\end{rem}

\begin{figure}[thbp]
\begin{tabular}{ccccc}

\begin{tikzpicture}
\draw[thin,gray] (0,0) circle (1);

\filldraw [black] (0:1) circle (1pt) node [anchor = west] {$1$};
\filldraw [black] (90:1) circle (1pt) node [anchor = south] {$3$};
\filldraw [black] (135:1) circle (1pt) node [anchor = south] {$2$};
\filldraw [black] (180:1) circle (1pt) node [anchor = east] {$5$};
\filldraw [black] (270:1) circle (1pt) node [anchor = north] {$4$};

\draw[thin, gray, dashed] (0:1)--(180:1);
\draw[thin, gray, dashed] (90:1)--(270:1);
\end{tikzpicture}

&
&
&
&

\begin{tikzpicture}
\draw[thin,gray] (0,0) circle (1);

\filldraw [black] (0:1) circle (1pt) node [anchor = west] {$1$};
\filldraw [black] (45:1) circle (1pt) node [anchor = south] {$3$};
\filldraw [black] (90:1) circle (1pt) node [anchor = south] {$2$};
\filldraw [black] (180:1) circle (1pt) node [anchor = east] {$5$};
\filldraw [black] (270:1) circle (1pt) node [anchor = north] {$4$};

\draw[thin, gray, dashed] (0:1)--(180:1);
\draw[thin, gray, dashed] (90:1)--(270:1);
\end{tikzpicture}

\end{tabular}
\caption{$X_5$ realized in the geodesic circle $S^1$.}
\label{fig:1}
\end{figure}
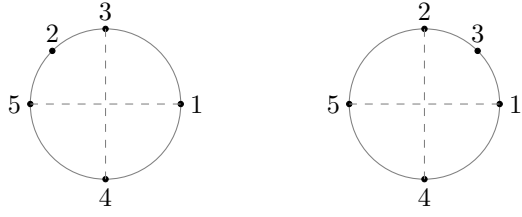

\smallskip

\begin{rem}
We have $b_- = b_0$ if $n = 3$, and $b_- \le b_0$ if $n = 4$ as shown in Lemma \ref{lem:b_0}. However, in the case that $n = 5$, there exists an example such that $b_0 < b_-$. Using the notations in Remark \ref{rem:converse_to_C2_implication}, we let $X_5 = \{ 1, 2, 3, 4, 5 \} \subset S^1$ be the subspace consisting of\begin{align*}
1 &= p(0), &
2 &= p(\pi/2), &
3 &= p(\pi/4), &
4 &= p(-\pi/2), &
5 &= p(\pi),
\end{align*}
as illustrated in Figure \ref{fig:1} (right). We have $b_- = Z_{13}Z_{23}$ because $Z_{12} = Z_{13}Z_{23}$. We also have $P < 0$ by $Z_{24} = Z_{23}Z_{34}, Z_{25} > Z_{23}Z_{35}$ and $Z_{45} > Z_{34}Z_{35}$. Finally, by $Z_{14} > Z_{13}Z_{34}$ and $Z_{15} = Z_{13}Z_{35}$, we see $b_0 < b_-$.
\end{rem}

\medskip

\begin{rem} \label{rem:Mayer_Vietoris}
\textit{Magnitude homology} \cite{L-S} is a notion which categorifies the magnitude of a finite metric space. If a finite metric space $(X, d)$ satisfies (C1) and (C2), then its magnitude homology fits into a (splitting) \textit{Mayer-Vietoris exact sequence} \cite{B-K,T-Y}. Generally, the Mayer-Vietoris sequence for the magnitude homology implies the inclusion-exclusion formula for the magnitude \cite{H-W,L-S}. Therefore a natural question is whether the magnitude homology of a finite metric space subject to $Z_{12} = b_0$ fits into the Mayer-Vietoris sequence. It turns out that $Z_{12} = b_0$ does not generally lead to the Mayer-Vietoris sequence. 

This is seen by an example: Consider the metric $d$ on $X_4 = \{ 1, 2, 3, 4 \}$ such that:
\begin{itemize}
\item
$d(i, j) = 1$ for distinct $i, j \in A = \{ 1, 3, 4 \}$, and also $d(i, j) = 1$ for distinct $i, j \in B = \{ 2, 3, 4 \}$.

\item
$d(1, 2) = \log\big(\frac{e^2 + e}{2}\big) = 1.6201 \cdots$.

\end{itemize}
Note that $b_- = e^{-2} < b_0 = Z_{12} = 2 e^{-2}/(1 + e^{-1})$, hence the inclusion-exclusion principle is satisfied. The subspaces $A$, $B$ and $A \cap B$ can be identified with the metric spaces associated to the complete graphs. Thus, on the one hand, the magnitude homology groups $MH^\ell_n(A)$, $MH^\ell_n(B)$ and $MH^\ell_n(A \cap B)$ of the subspaces are trivial for all $n \in \Z$ provided that $\ell = d(1, 2)$. On the other hand, we readily see that $MH^\ell_1(X_4) \cong \Z^2$ for $\ell = d(1, 2)$. Therefore we just get a sequence which is not exact:
$$
\cdots \to
\underbrace{MH^\ell_1(A \cap B)}_0 \to
\underbrace{MH^\ell_1(A)}_0 \oplus \underbrace{MH^\ell_1(B)}_0 \to
\underbrace{MH^\ell_1(X_4)}_{\Z^2} \to
\underbrace{MH^\ell_0(A \cap B)}_0 \to
\cdots.
$$
\end{rem}



\begin{thebibliography}{999}


\bibitem{B-K}R.~Bottinelli and T.~Kaiser, 
\textit{Magnitude homology, diagonality, and median spaces}.
Homology Homotopy Appl. 23 (2021), no. 2, 121--140.



\bibitem{D-S}A.~W.~M.~Dress and R.~Scharlau,
\textit{Gated sets in metric spaces}.
Aequationes Math. 34 (1987), no. 1, 112--120.



\bibitem{H-W}R.~Hepworth and S.~Willerton,
\textit{Categorifying the magnitude of a graph}.
Homology Homotopy Appl. 19 (2017), no. 2, 31--60.


\bibitem{L2}T.~Leinster,
\textit{The magnitude of a graph}.
Math. Proc. Cambridge Philos. Soc. 166 (2019), no. 2, 247--264.



\bibitem{L1}T.~Leinster,
\textit{The magnitude of metric spaces}.
Doc. Math. 18 (2013), 857--905.


\bibitem{L-S}
T.~Leinster and M.~ Shulman,
\textit{Magnitude homology of enriched categories and metric spaces}.
Algebr. Geom. Topol. 21 (2021), no. 5, 2175--2221.


\bibitem{M}M.~W.~Meckes,
\textit{Positive definite metric spaces}.
Positivity 17 (2013), no.3, 733--757.


\bibitem{T-Y}Y.~Tajima and M.~Yoshinaga,
\textit{Causal order complex and magnitude homotopy type of metric spaces}.
Internat. Math. Res. Notices (2023), 
arXiv:2302.09752.


\end{thebibliography}
\end{document}